\documentclass[12pt]{amsart}
\usepackage[active]{srcltx}
\usepackage{calc,amssymb,amsthm,amsmath,amscd, eucal,ulem}
\usepackage{alltt}
\RequirePackage[dvipsnames,usenames]{color}

\input{mabliautoref.sty}
\input{xy}
\xyoption{all}
\input{kmacros3.sty}
\usepackage{tikz}
\usepackage{amsfonts, mathrsfs}
\usepackage{calligra}
\usepackage{stmaryrd} 
\usepackage[left=1.1in,top=1in,right=1.1in,bottom=1in]{geometry}
\usepackage{mathtools}
\numberwithin{equation}{theorem}

\renewcommand{\m}{\mathfrak{m}}

\newcommand{\eHK}{e_{\mathrm{HK}}}

\DeclareMathOperator{\frk}{frk}

\usepackage{setspace}
\usepackage{hyperref}
\usepackage{enumerate}
\usepackage{graphicx}
\usepackage[all,cmtip]{xy}

\usepackage{verbatim}

\theoremstyle{theorem}

\renewcommand{\O}{\mathscr O}

\begin{document}
\title{Bertini theorems for $F$-signature and Hilbert--Kunz multiplicity}
\author{Javier Carvajal-Rojas}
\author{Karl Schwede}
\author{Kevin Tucker}
\address{\'Ecole Polytechnique F\'ed\'erale de Lausanne, SB MATH CAG, MA C3 615 (B\^atiment MA),
Station 8, CH-1015 Lausanne, Switzerland\newline\indent
Escuela de Matem\'atica\\ Universidad de Costa Rica\\ San Jos\'e 11501\\ Costa Rica}
\email{javier.carvajalrojas@epfl.ch}
\address{Department of Mathematics\\ University of Utah\\ Salt Lake City\\ UT 84112}
\email{schwede@math.utah.edu}
\address{Department of Mathematics\\ University of Illinois at Chicago\\Chicago\\  IL 60607}
\email{kftucker@uic.edu}

\thanks{The first named author was supported in part by the NSF FRG Grant DMS \#1265261/1501115 and NSF CAREER Grant DMS \#1252860/1501102 and by the ERC-STG \#804334.}
\thanks{The second named author was supported in part by the NSF FRG Grant DMS \#1265261/1501115, NSF CAREER Grant DMS \#1252860/1501102 and NSF Grant DMS \#1801849.}
\thanks{The third named author was supported in part by NSF Grants DMS \#1602070 and \#1707661 and a fellowship from the Sloan foundation.}


\maketitle

\begin{abstract}
We show that Bertini theorems hold for $F$-signature and Hilbert--Kunz multiplicity.  
In particular, if $X \subseteq \bP^n$ is normal and quasi-projective with $F$-signature greater than $\lambda$ (respectively the Hilbert--Kunz multiplicity is less than $\lambda$) at all points $x \in X$, then for a general hyperplane $H \subseteq \bP^n$ the $F$-signature (respectively Hilbert--Kunz multiplicity) of $X \cap H$ is greater than $\lambda$ (respectively less than $\lambda$) at all points $x \in X \cap H$.
\end{abstract}

\section{Introduction}

A common tool for studying a quasi-projective algebraic variety $X \subseteq \bP^n_k$, $k = \overline{k}$, is to perform induction on dimension by intersecting with a general hyperplane $H$.  When doing this, we want the resulting intersection $X \cap H$ to have similar properties to the original variety $X$.  Bertini's theorem accomplishes exactly this: the classical result asserts that if $X$ is smooth then so is $X \cap H$ for a general choice of $H$ \cite[II, Theorem 8.18]{Hartshorne}, \cite{KleimanBertiniAndHisTheorems}. Many classes of singularities also satisfy this property.  For example, in characteristic zero, if $X$ is log terminal (respectively log canonical), then so is $X \cap H$ \cite[Lemma 5.17]{KollarMori}.  Even more generally the multiplier ideal of a divisor pair restricts to the multiplier ideal of the intersection
\[
\mJ(X, \Delta) |_{X \cap H} = \mJ(X \cap H, \Delta_{X \cap H}),
\]
see \cite[Example 9.5.9]{LazarsfeldPositivity2}.  In characteristic zero, Bertini theorems can be generalized to the case where $H$ is a general member of a base point free linear system.

In characteristic $p > 0$, the situation is more complicated.  It is essential that $H$ is a general member of a very ample linear system (or something close to that) if you expect Bertini-type results to hold.  Since strongly $F$-regular and $F$-pure singularities are analogous to log terminal and log canonical singularities respectively \cite{HaraWatanabeFRegFPure}, it is natural to expect that the corresponding Bertini-results hold.  In \cite{SchwedeZhangBertiniTheoremsForFSings}, this is exactly what was shown.
\begin{theorem*}[\cite{SchwedeZhangBertiniTheoremsForFSings}]
If $(X, \Delta)$ is a strongly $F$-regular (resp. sharply $F$-pure) pair such that $X \subseteq \bP^n_k$ is quasi-projective and $k = \overline{k}$ is of characteristic $p > 0$, then $(X \cap H, \Delta|_{X \cap H})$ is also strongly $F$-regular (resp. sharply $F$-pure) for a general choice of hyperplane $H \subseteq \bP^n_k$.
\end{theorem*}
\noindent
However, the corresponding result for test ideals is false:
\begin{theorem*}[\cite{BydlonCounterExamplesToBertiniTestIdeals}]
For any $p > 0$ and $n \geq 3$, there exists a $\bQ$-divisor $\Delta$ on $X = \bA^n_k$, where $k = \overline{k}$ is of characteristic $p > 0$, such that
\[
\tau(X, \Delta)|_{H} \neq \tau(X \cap H, \Delta|_{X \cap H})
\]
for a general hyperplane $H \subseteq \bA^n$.
\end{theorem*}
It is then natural to ask about other types of $F$-singularities in characteristic $p > 0$.  For example, the behavior of $F$-rational singularities under restriction to general hyperplanes is still unknown.  In this paper, we show that the above sort of Bertini-theorem holds for $F$-signature $s(\O_{X,x})$ and Hilbert--Kunz multiplicity $\eHK(\O_{X,x})$ in the following sense.
\begin{mainthm*}[\autoref{thm.Main}, \autoref{thm.MainHK}]
Suppose that $X \subseteq \bP^n_k$ is a normal quasi-projective variety, $k = \overline{k}$ is of characteristic $p > 0$, and $\Delta \geq 0$ is a $\bQ$-divisor.  Suppose that $\lambda \geq 0$ is a number such that the $F$-signature is bigger than $\lambda$,
\[
s(\O_{X,x}, \Delta) > \lambda,
\]
for all $x \in X$.  Then for a general hyperplane $H \subseteq \bP^n_k$,
\[
s(\O_{X \cap H, x}, \Delta|_{X \cap H}) > \lambda
\]
for all $x \in X \cap H$.

{
Similarly suppose that $\lambda \geq 1$ is a number such that the Hilbert--Kunz multiplicity is less than $\lambda$,
\[
\eHK(\O_{X,x}) < \lambda,
\]
for all $x \in X$.  Then for a general hyperplane $H \subseteq \bP^n_k$,
\[
\eHK(\O_{X \cap H, x}) < \lambda
\]
for all $x \in X \cap H$.}
\end{mainthm*}
\noindent
We actually prove a slightly stronger result by weakening the hypothesis that $X \subseteq \bP^n_k$ and we also make statements about the locus $U$ where $s(\O_{X,x}, \Delta) > \lambda$ for all $x \in U$ or likewise with the locus where $\eHK(\O_{X,x}) < \lambda$.

Recall that $F$-signature measures how strongly $F$-regular a variety or pair is.  Explicitly, if $R$ is finite type over $k = \overline{k}$, then $R$ is regular if and only if $R^{1/p^e}$ is a locally free $R$-module by \cite{KunzCharacterizationsOfRegularLocalRings}.  The $F$-signature refines this. By definition, $s(R)$ is a number that indicates what percentage of $R^{1/p^e}$ is locally free asymptotically as $e$ goes to $\infty$.
Thus $1 \geq s(R) \geq 0$ and
\begin{itemize}
\item{} $s(R) = 1$ if and only if $R$ is regular \cite{HunekeLeuschkeTwoTheoremsAboutMaximal} (\cf \cite{YaoModulesWithFFRT}) and
\item{} $s(R) > 0$ if and only if $R$ is strongly $F$-regular \cite{AberbachLeuschke}.
\end{itemize}
The $F$-signature should be thought of some sort of local volume of the singularity.

{
On the other hand, Hilbert--Kunz measures how close a ring is to being regular.  If $(R, \fram, k = k^p)$ is a local ring of dimension $d$, then $\eHK(R)$ is the asymptotic value of the ratio between the number of generators of $R^{1/p^e}$ as an $R$-module with the number of generators expected for a regular ring ($p^{ed}$).  One has that
$e_{\mathrm{HK}}(R) \geq 1$ and
\begin{itemize}
\item{} $\eHK(R) = 1$ if and only if $R$ is regular \cite[Theorem 1.5]{WatanabeYoshidaHKMultAndInequality}.
\end{itemize}
Hilbert--Kunz multiplicity is another sort of volume of a singularity.
}

We prove our main result by relying on the axiomatic Bertini framework as introduced in \cite{CuminoGrecoManaresiAxiomatic}.  In particular, to show the type of result in our Main Theorem, it suffices to show the following two properties for a property of singularities $\sP$ (such as $s(\O_{X,x}) > \lambda$):
\begin{itemize}
\item[(A1)]  If $\phi : Y \to Z$ is a flat morphism with regular fibers and $Z$ is $\sP$, then $Y$ is $\sP$ too.
\item[(A2)]  Let $\phi : Y \to S$ be a morphism of finite type where $Y$ is excellent and $S$ is integral with generic point $\eta$.  If $Y_{\eta}$ is geometrically $\sP$, then there exists an open neighborhood $U$ of $\eta$ in $S$ such that the fibers $Y_{s}$ are geometrically $\sP$ for each $s \in U$.  (In fact, it suffices to check this for $S = (\bP^n_k)^*$, the space of hyperplanes).
\end{itemize}
Property (A1) was already proven for $F$-signature in \cite{YaoObservationsAboutTheFSignature}. In \autoref{sec.YaoProofForPairs}, we generalize this result to the context of pairs and give a new proof in the classical non-pair setting.   In \autoref{sec.FsignatureOfGeneralFibers}, we show that property (A2) holds for $F$-signature.
\vskip 9pt
\noindent
{\it Acknowledgements:} The authors thank Patrick Graf and Yongwei Yao for stimulating discussions.  Work on this project was conducted in CIRM (Luminy) and Oberwolfach. They also thank the anonymous referee for very valuable comments and suggestions.

\section{Preliminaries}

\subsection{Hilbert--Kunz multiplicity and $F$-signature}
Throughout this article, we shall assume all schemes $X$ are Noetherian, separated, and have prime characteristic $p>0$.  If $x \in X$, we let $k(x)$ denote the residue field of the local ring $\O_{X,x}$.  We let $F^e \colon X \to X$ denote the $e$-iterated Frobenius endomorphism or $p^e$-th power map. We say $X$ is $F$-finite if $F^e$ is a finite morphism, in which case $X$ is automatically excellent and has a dualizing complex \cite{KunzOnNoetherianRingsOfCharP,Gabber.tStruc}.

When $X = \Spec(A)$ is affine, we often conflate scheme-theoretic and ring-theoretic notation.   In particular, $F^e \colon A \to A$  denotes the $e$-iterated Frobenius, and for an $A$-module $M$ we write $F^e_*M$ for $\Gamma\bigl(\Spec(A), F^e_*\widetilde{M}\bigr)$ where $\widetilde{M}$ is the associated quasi-coherent sheaf on $\Spec(A)$.   In other words, $F^e_*M$ is the $A$-module arising from $M$ via restriction of scalars for $F^e$.  In case $A$ is reduced, we also identify $F^e$ with the inclusion $A \subseteq A^{1/p^e}$, and  shall at times use $M^{1/p^e}$ to denote $F^e_*M$ accordingly.

If $J \subseteq A$ is an ideal, then the $e$-th Frobenius power of $J$ is  the expansion of $J$ under the $e$-iterated Frobenius and denoted $J^{[p^e]} = \langle F^e(J) \rangle = \bigl\langle j^{[p^e]} \mid j \in J \bigr\rangle$. It follows $J  (F^e_* M) = F^e_*\bigl(J^{[p^e]}  M\bigr)$ or $J \bigl(M^{1/p^e}\bigr) = \bigl(J^{[p^e]}M\bigr)^{1/p^e}$ for any $A$-module $M$. In the local setting, the Frobenius powers give rise to the following well-studied variant on the Hilbert--Samuel multiplicity.

\begin{definition}
 If $(A, \m)$ is a local ring of dimension $d$, the Hilbert--Kunz multiplicity of $A$ is
 \[
e_{\mathrm{HK}}(A) = \lim_{e \rightarrow \infty} \frac{1}{p^{ed}} \ell_A\bigl( A / \m^{[p^e]} \bigr),
\]
where we write $\ell_A( \blank)$ for the length of an $A$-module
\end{definition}

\begin{theorem}
Suppose $(A, \m)$ is a local ring of dimension $d$.
\begin{enumerate}
\item
\cite{MonskyHKFunction}
The limit defining the Hilbert--Kunz multiplicity $e_{\mathrm{HK}}(A)$ exists, and moreover
\[
\ell_A\bigl( A / \m^{[p^e]} \bigr) = e_{\mathrm{HK}}(A) \cdot p^{ed} + O\bigl(p^{e(d - 1)}\bigr).
\]
\item
\cite{WatanabeYoshidaHKMultAndInequality}
The Hilbert--Kunz multiplicity $e_{\mathrm{HK}}(A) \geq 1$, and if $A$ is equidimensional then $e_{\mathrm{HK}}(A) = 1$ if and only if $A$ is regular.
\end{enumerate}
\end{theorem}

The $F$-signature, like the Hilbert--Kunz multiplicity, is another important numerical invariant of a local ring in positive characteristic defined in terms of the iterates of Frobenius.  For any positive characteristic ring $A$,
recall that an $A$-module inclusion $M_1 \to M_2$ is said to be pure if $M_1 \otimes_A N \to M_2 \otimes_A N$ remains injective for any $A$-module $N$.  An inclusion $A \to M$, where $M$ is a finitely generated $A$-module, is pure if and only if it is split, \textit{i.e.} admits an $A$-module section; see \cite[Corollary 5.2]{HochsterRobertsFrobeniusLocalCohomology}.  If $(A,\m)$ is local, $A \to M$ is pure if and only if $E_A(k) \to M \otimes_A E_A(k)$ is injective, where $E_A(k)$ is an injective hull of the residue field $k = A/\m$; this follows from Matlis duality \cite[Lemma 2.1 (e)]{HochsterHunekeApplicationsExistenceBCMAlgebras}. We write $\ell_A( \blank)$ for the length of an $A$-module, omitting the subscript at times to simplify  notation.

\begin{definition}
If $(A,\m)$ is an excellent local ring of dimension $d$, the $e$-th Frobenius degeneracy ideal
\[
I_e(A) = \bigl\langle a \in A \mid A \xrightarrow{1 \mapsto F^e_* a}  F^e_*A \mbox{ is not a pure $A$-module inclusion} \bigr\rangle
\]
is an ideal of $A$, and the $F$-signature is
\[
s(A) = \lim_{e \rightarrow \infty} \frac{1}{p^{ed}} \ell_A\bigl( A / I_e(A) \bigr).
\]
\end{definition}

Recall the following results on $F$-signature.

\begin{theorem}
Suppose $(A, \m)$ is an excellent local ring of dimension $d$.
\begin{enumerate}
\item \cite{TuckerFSigExists}
The limit defining the $F$-signature $s(A)$ exists, and moreover
\[
\ell_A\bigl( A / I_e(A) \bigr) = s(A) \cdot p^{ed} + O\bigl(p^{e(d - 1)}\bigr).
\]
\item \cite{HunekeLeuschkeTwoTheoremsAboutMaximal}
The $F$-signature $s(A) \leq 1$, and $s(A) = 1$ if and only if $A$ is regular.
\item \cite{AberbachLeuschke, YaoModulesWithFFRT}
The $F$-signature $s(A) \geq 0$, and $s(A) > 0$ if and only if $A$ is strongly $F$-regular.  In this case, $A$ is necessarily a Cohen--Macaulay normal domain.
\end{enumerate}
\end{theorem}

The Hilbert--Kunz multiplicity and $F$-signature are also known to satisfy additional properties in the $F$-finite setting, such as semi-continuity.

\begin{theorem} \cite{SmirnovSemiContHK, PolstraFsigSemiCont,PolstraTuckerFSingHKMultCombinedApproach}
Consider an $F$-finite domain $A$.\footnote{Note that upper semi-continuiuty of the Hilbert--Kunz multiplicity is also known to hold for a ring which is essentially of finite type over an excellent local ring.}
\begin{enumerate}
\item
The Hilbert--Kunz mulitiplicity determines an upper semi-continuous function
\[
Q \in \Spec(A) \mapsto e_{\mathrm{HK}}(A_Q)
\]
on $\Spec(A)$.
\item
The $F$-signature determines a lower semi-continuous function
\[
Q \in \Spec(A) \mapsto s(A_Q)
\]
on $\Spec(A)$.
\end{enumerate}
\end{theorem}

\noindent
Moreover, if $(A,\m)$ is an $F$-finite local ring of dimension $d$, note that one can alternately describe the degeneracy ideals as
\begin{align*}
I_e(A) &= \bigl\langle a \in A \mid A \xrightarrow{1 \mapsto F^e_* a}  F^e_*A \mbox{ is not a split $A$-module inclusion} \bigr\rangle \\
&= \bigl\langle a \in A \mid \phi(F^e_*a) \in \m \mbox{ for all } \phi \in \Hom_A(F^e_*A,A) \bigr\rangle,
\end{align*}
and the $F$-signature can be viewed as giving an asymptotic measure of the number of splittings of the $e$-iterated Frobenius.
In particular, if $(A, \m)$ is an $F$-finite local domain, we have
\[
e_{\mathrm{HK}}(A) = \lim_{e \rightarrow \infty} \frac{\mu_A \bigl(A^{1/p^e}\bigr)}{\rank_A (A^{1/p^e})}
\quad
\mbox{ and }
\quad
s(A) = \lim_{e \rightarrow \infty} \frac{\frk_A \bigl(A^{1/p^e}\bigr)}{\rank_A (A^{1/p^e})}.
\]
where $\mu_A(\blank)$ denotes the minimal number of generators and $\frk_A(\blank)$ denotes free rank; see \cite{PolstraTuckerFSingHKMultCombinedApproach} for details.  Recall that, for arbitrary (and not necessarily local) $A$, the free rank  of an $A$-module $M$ is the maximal rank $\frk_A(M)$ of a free $A$-module quotient of $M$.  \emph{For us going forward, we will use $a_e(R)$ (resp. $b_e(R)$) to denote the free rank (resp. minimal number of generators) of $R^{1/p^e}$, and may write simply $a_e$ (resp. $b_e$) if the context is clear}

One can generalize the interpretation of Hilbert--Kunz multiplicity and $F$-signature for $F$-finite rings beyond the local setting as well. To make this more precise, recall first the following result of Kunz.

\begin{lemma}\cite{KunzOnNoetherianRingsOfCharP}
If $A$ is a reduced equidimensional $F$-finite ring, the function
\[
Q \in \Spec(A) \mapsto \bigl[k(Q)^{1/p^e}:k(Q)\bigr] \cdot p^{e\height Q}
\]
is constant on $\Spec(A)$.  In particular, if $A$ is a domain,
\[
\rank_A\bigl(A^{1/p^e}\bigr)  = \bigl[k(Q)^{1/p^e}:k(Q)\bigr] \cdot p^{e \height Q}
\]
for any $e \geq 0$ and $Q \in \Spec(A)$.
\end{lemma}

\noindent
We recall a recent result globalizing Hilbert--Kunz multiplicity and $F$-signature.

\begin{theorem}\cite{DeStefaniPolstraYaoGlobalizingFinvariants}
If $A$ is a reduced equidimensional $F$-finite ring, and $\gamma \in \ZZ_{\geq 0}$ with  $p^{\gamma} = \bigl[k(Q)^{1/p}:k(Q)\bigr] \cdot p^{\height Q} $ for all $Q \in \Spec(A)$, then
the limit
\[
e_{\mathrm{HK}}(A) = \lim_{e \rightarrow \infty} \frac{\mu_A\bigl(A^{1/p^e}\bigr)}{p^{e \gamma}}
\]
exists and equals $\max \{ e_{\mathrm{HK}}(A_Q) \mid Q \in \Spec(A) \} = \max \{ e_{\mathrm{HK}}(A_\m) \mid \m \in \max \Spec(A) \}$.
Similarly,
the limit
\[
s(A) = \lim_{e \rightarrow \infty} \frac{\frk_A\bigl(A^{1/p^e}\bigr)}{p^{e \gamma}}
\]
exists and equals $\min \{ s(A_Q) \mid Q \in \Spec(A) \} = \min \{ s(A_\m) \mid \m \in \max \Spec(A) \}$.
\end{theorem}

\subsection{Divisors}
In this subsection, we review the definitions and properties of the $F$-signature of divisor pairs.

\begin{definition}
If $(A,\m)$ is a normal excellent local domain of dimension $d$ and $D$ is an effective Weil divisor on $\Spec(A)$, the $e$-th Frobenius degeneracy ideal along $D$ is
\[
I_e(A,D) = \bigl\langle a \in A \mid A \xrightarrow{1 \mapsto F^e_* a}  F^e_*\bigl(A (D)\bigr)\mbox{ is not a pure $A$-module inclusion} \bigr\rangle.
\]
 If $\Delta$ is an effective $\Q$-divisor on $\Spec(A)$, the $F$-signature of $(A, \Delta)$ is
\[
s(A, \Delta) = \lim_{e \rightarrow \infty} \frac{1}{p^{ed}} \ell_A\Bigl( A \big/ I_e\bigl(A, \lceil (p^e - 1) \Delta \rceil \bigr) \Bigr).
\]
\end{definition}

\begin{lemma}
\label{lem:perturb}
Suppose $(A,\m)$ is a normal excellent local domain of dimension $d$ and $\Delta$ is an effective $\Q$-divisor on $\Spec(A)$.
Let $\{ D_e \}_{e > 0}$ be a sequence of Weil divisors on $\Spec(A)$ with bounded difference from $\bigl\{ \lceil (p^e -1) \Delta \rceil \bigr\}_{e > 0}$ independent of $e > 0$.  In other words, there exists an effective Cartier divisor $C$ such that
\[
- C \leq D_e - \lceil (p^e -1) \Delta \rceil \leq C
\]
for all $e > 0$.  Then
\[
s(A, \Delta) = \lim_{e \rightarrow \infty} \frac{1}{p^{ed}} \ell_A\bigl( A / I_e(A, D_e) \bigr).
\]
\end{lemma}
\begin{proof}
This is essentially the same argument as \cite[Lemma 4.17]{BlickleSchwedeTuckerFSigPairs1} and \cite[Theorem 4.13]{PolstraTuckerFSingHKMultCombinedApproach}, and so we omit it.
\end{proof}


\begin{theorem}\cite{BlickleSchwedeTuckerFSigPairs1,PolstraTuckerFSingHKMultCombinedApproach}
Suppose $(A,\m)$ is a normal excellent local domain of dimension $d$ and $\Delta$ is an effective $\Q$-divisor on $\Spec(A)$.
\begin{enumerate}
\item
The limit defining the $F$-signature $s(A, \Delta)$ exists, and moreover
\[
\ell_A\Bigl( A \big/ I_e\bigl(A, \lceil (p^e - 1) \Delta \rceil \bigr) \Bigl) = s(A, \Delta) \cdot p^{ed} + O\bigl(p^{e(d - 1)}\bigr).
\]
\item
The $F$-signature $s(A, \Delta) \geq 0$, and $s(A, \Delta) > 0$ if and only if $(A, \Delta)$ is strongly $F$-regular.
\end{enumerate}
\end{theorem}

If $A$ is an $F$-finite normal excellent domain of dimension $d$ and $D$ is an effective Weil divisor on $\Spec(A)$, one can define the \emph{free rank of $A^{1/p^e}$ along $D$}
\begin{equation} \label{eqn.Definition:a_e^D}
    a_e^{D}(A) = \frk_{A}^D\bigl(A^{1/p^e}\bigr)
\end{equation}

to be the maximal rank $a_e(D)$ of a simultaneous free $A$-module quotient of $A^{1/p^e}$ and $\bigl(A(D)\bigr)^{1/p^e}$.    In other words, $a_e(D)$ is the largest non-negative integer such that there is a commuting diagram
\[
\xymatrix{
& \bigl(A(D)\bigr)^{1/p^e} \ar@{->>}[dr] & \\
A^{1/p^e} \ar@{->>}[rr] \ar[ru]^{\subseteq} & & A^{\oplus a_e(D)}.
}
\]
In case $(A,\m)$ is local, we have that $\frk_{A}^D\bigl(A^{1/p^e}\bigr) = [k(\fram)^{1/p^e}:k(\fram)]\cdot \ell_A\bigl( A / I_e(A, D) \bigr)$ (see \cite[Proposition 3.5]{BlickleSchwedeTuckerFSigPairs1}), and once more this leads to a recent global interpretation of the $F$-signature along a divisor.

\begin{theorem} \cite{DeStefaniPolstraYaoGlobalizingFinvariants}
\label{thm.globaldivisorfsig}
Let $A$ be an $F$-finite normal excellent  domain of dimension $d$, and $\gamma \in \ZZ_{\geq 0}$ with  $p^{\gamma} = \bigl[k(Q)^{1/p}:k(Q)\bigr] \cdot p^{\height Q} $ for all $Q \in \Spec(A)$.   Suppose $\Delta$ is an effective $\Q$-divisor on $\Spec(A)$.
The $F$-signature along $\Delta$ determines a lower semi-continuous function
\[
Q \in \Spec(A) \mapsto s(A_Q, \Delta)
\]
on $\Spec(A)$.  Moreover, the limit
\[
s(A, \Delta) = \lim_{e \rightarrow \infty} \frac{\frk_A^{\lceil p^e \Delta \rceil}\bigl(A^{1/p^e}\bigr)}{p^{e \gamma}}
\]
exists and equals $\min \{ s(A_Q, \Delta) \mid Q \in \Spec(A) \} = \min \{ s(A_\m, \Delta) \mid \m \in \max \Spec(A) \}$.
\end{theorem}

\noindent
In light of \autoref{thm.globaldivisorfsig}, and following \cite{DeStefaniPolstraYaoGlobalizingFinvariants},
we also make the following global definition.

\begin{definition}
For a normal $F$-finite scheme $X$ and effective $\bQ$-divisor $\Delta$ we set
\[
s(X, \Delta) = \min \{ s(\O_{X,x}, \Delta) \mid x \in X \} = \min \{ s(\O_{X,x}, \Delta) \mid x \in X \mbox{ a closed point} \}.
\]
When $X = \Spec A$ is affine, we write $s(A, \Delta)$ for $s(X, \Delta)$.
\end{definition}

\subsection{Divisors and families}
\label{subsec.DivisorsAndFamilies}
Finally, we discuss the correspondence between $\bQ$-divisors and $p^{-e}$-linear maps in the relative setting of $A \subseteq R$ (or in other words, for families).  What follows is contained in \cite{PatakfalviSchwedeZhangFFamilies} although we work in a less general setting.

\begin{setting}
\label{set.RelativeSetting}
Suppose that $A$ is an $F$-finite regular domain and suppose we have $A \subseteq R$ a flat finite type extension of rings with geometrically\footnote{Here we mean that the fibers are normal after any base change, including inseparable ones.} normal fibers.  Additionally assume that for some choice of $\omega_A$,
\begin{equation}
\label{eq.FShriekCompatible}
\tag{$\dagger$}
F^! \omega_A \cong \omega_A.
\end{equation}
This always holds for rings essentially of finite type over a Gorenstein semi-local ring.

For any $A$-algebra $B$, we write $R_{B} = R \otimes_A B$.  Frequent values of $B$ include $A^{1/p^e}$, the fraction field $K \coloneqq K(A)$ and $k(Q)$, the residue field of a point $Q \in \Spec A$.
\end{setting}

We make some quick observations.

\begin{lemma}
\label{lem.BaseChangeIsNormalInSetting}
In the setting of \autoref{set.RelativeSetting}, each $R_{A^{1/p^e}}$ is a normal integral domain, as are $R_{K^{1/p^e}}$ and $R_{K^{\infty}}$ as well.
\end{lemma}
\begin{proof}
$A^{1/p^e} \to R_{A^{1/p^e}}$ is flat with normal fibers over a regular base, and hence $R_{A^{1/p^e}}$ is normal by \cite[Theorem 23.9]{MatsumuraCommutativeRingTheory}.
Since $R \to  R_{A^{1/p^e}}$ is purely inseparable and $R_{A^{1/p^e}}$ is reduced, it follows that $R_{A^{1/p^e}}$ is a domain.  Localizing, we have that $K \to R_{K}$ also has geometrically normal fibers, and the same argument gives that $R_{K^{1/p^e}}$ and $R_{K^\infty}$ are normal domains as well.
%
\end{proof}

\begin{lemma}
In the setting of \autoref{set.RelativeSetting}, for each $Q \in \Spec A$ and $x \in \Spec R_{K(Q)} \subseteq \Spec R$ a point of codimension $1$ on the fiber, we have that $R_x$ is regular and thus $\Delta$ is $\bQ$-Cartier at $x$.  In particular, we can restrict  $\Delta|_{\Spec R_{k(Q)}}$ to any fiber.
\end{lemma}
\begin{proof}
Choose a codimension 1 point $x \in \Spec R_{K(Q)}$,  in other words a codimension one point of a fiber over $Q \in \Spec(A)$.  In particular, $(R_{K(Q)})_x $ is normal and hence regular.  It follows that $R_x$ is also regular since $R_{K(Q)}$ is obtained from $R$ by killing a regular sequence and localizing.
\end{proof}

We now discuss the correspondence between divisors and maps in \autoref{set.A2Setting}.

\begin{lemma}\textnormal{\cite[2.8--2.11]{PatakfalviSchwedeZhangFFamilies}}
\label{lem.RelativeDivisorMapCorresponds}
Suppose that $A$ is an $F$-finite regular domain and suppose we have $A \subseteq R$ a flat finite type extension of rings with geometrically
normal fibers.  Then for every $R_{A^{1/p^e}}$-linear map
\[
\phi : R^{1/p^e} \to R_{A^{1/p^e}}
\]
which generates $\Hom_{R_{A^{1/p^e}}}(R^{1/p^e}, R_{A^{1/p^e}})$ at the generic point of every fiber, there exists a corresponding $\bZ_{(p)}$-divisor\footnote{A $\bQ$-divisor in which no denominators contain $p$.} on $\Spec R$
\[
\Delta_{\phi} \sim_{\bQ} -K_{R/A}
\]
which does not contain any fiber in its support.

Conversely, given an effective $\bZ_{(p)}$-divisor $\Delta \sim_{\bQ} -K_{R/A}$ on $\Spec R$ whose support does not contain any fiber, we can construct a map $\phi : R^{1/p^e} \to R_{A^{1/p^e}}$ such that $\Delta_{\phi} = \Delta$.
\end{lemma}

Finally, we recall the interaction between divisors and maps behaves under base change.  While not crucial for the following statement, in this paper we restrict ourselves to base changes which are either flat or restriction to a fiber followed by a flat base change, which is easier to work with than the generality of \cite{PatakfalviSchwedeZhangFFamilies}.

\begin{lemma}\textnormal{\cite[Lemma 2.21]{PatakfalviSchwedeZhangFFamilies}}
\label{lem.RelativeDivisorsAndBaseChange}
In the setting of \autoref{set.RelativeSetting} assume that $\Delta = \Delta_{\phi}$ is constructed as in \autoref{lem.RelativeDivisorMapCorresponds}.  For any regular $A$-algebra $B$ satisfying \eqref{eq.FShriekCompatible}, let $\pi : \Spec R_B \to \Spec R$ denote the canonical map.  Set $\phi_B \coloneqq \phi \otimes_{A^{1/p^e}} B^{1/p^e}$ to be the base changed map
\[
\phi_B : (R_B)^{1/p^e} = R^{1/p^e} \otimes_{A^{1/p^e}} B^{1/p^e} \to R_{A^{1/p^e}}\otimes_{A^{1/p^e}} B^{1/p^e} = R_{B^{1/p^e}}.
\]
In this case,
\[
\Delta_{\phi_B} = \pi^* \Delta = \pi^* \Delta_{\phi}.
\]
\end{lemma}

\begin{remark}
Frequently $B = A^{1/p^d}$ in which case the based changed map $\phi_B$ in \autoref{lem.RelativeDivisorsAndBaseChange} is simply
\[
\phi_{A^{1/p^d}} : \bigl(R_{A^{1/p^d}}\bigr)^{1/p^e} \to R_{A^{1/p^{e+d}}}.
\]
\end{remark}

\section{$F$-signature transformation for regular fibers}
\label{sec.YaoProofForPairs}

In this section, we will be concerned with the behavior of the $F$-signature under flat local extensions, building on the following result of Y.~Yao.

\begin{theorem}[\cite{YaoObservationsAboutTheFSignature}]
Suppose that $(A, \fram) \subseteq (R, \frn)$ is a flat local extension of excellent local rings of characteristic $p > 0$.  Then if $R/\fram R$ is regular, we have
\[
s(A) = s(R).
\]
\end{theorem}

\noindent
Our goal is to generalize the above result to the context of divisor pairs $(R, \Delta)$, for which we will first need to give a variation on the proof of the original statement. We begin with some preliminary lemmas.

\begin{lemma}
\label{lem.FlatLocalRegSeq}
Suppose that $(A,\fram) \subseteq (R, \frn)$ is a flat local extension of local rings.  If $x_1, \ldots, x_\delta \in R$ is a regular sequence on $R / \fram R$, then $ R / \langle x_1, \ldots, x_\delta \rangle$ is a flat $A$-algebra.  Moreover, $x_1, \ldots, x_\delta \in R$ are a regular sequence on $M \otimes_A R$ for any finitely generated $A$-module $M$, and lastly for any $t \geq 0$ the $R$-module inclusion
\[
R / \langle x_1^t, \ldots, x_\delta^t \rangle \xrightarrow{1 \mapsto \left[\frac{1}{x_1^t\cdots x_\delta^t} \right]} H^\delta_{\langle x_1, \ldots, x_\delta \rangle}(R)
\]
is pure as an inclusion of $A$-modules.
\end{lemma}

\begin{proof}
See \cite[Corollary 20.F, page 151]{MatsumuraCommutativeAlgebra} or \cite[Lemma 7.10]{HochsterHunekeFRegularityTestElementsBaseChange}.
For the final statement, note that it suffices to check purity after tensoring with finitely generated $A$-modules, where injectivity follows from the previous regular sequence assertion.
\end{proof}

The following was used in Hochster and Huneke's original study of $F$-regularity and base change.

\begin{lemma} \cite[Lemma 7.10]{HochsterHunekeFRegularityTestElementsBaseChange}
\label{lem.HHTensorLocalCMBaseChange}
Let $(A,\fram) \subseteq (R, \frn)$ be a flat local extension of local rings.  Suppose $R / \fram R$ is regular and $x_1, \ldots, x_\delta \in R$ give a regular system of parameters of $R / \fram R$.  If $E_A$ is an injective hull of $A / \fram$ over $A$ with socle generated by $u$, then $E_R = H^\delta_{\langle x_1, \ldots, x_\delta \rangle} (R) \otimes_A E_A$ is an injective hull of $R / \frn$ over $R$ with socle generated by $\Big[\frac{1}{x_1\cdots x_\delta}\Big] \otimes u$.
\end{lemma}

Now, we give a new proof of Yao's result.

\begin{proof}[Proof of Theorem 3.1]
If $x_1, \ldots, x_\delta \in R$ give a regular system of parameters of $R / \fram R$, then by \autoref{lem.HHTensorLocalCMBaseChange} we have $E_R = H^\delta_{\langle x_1, \ldots, x_\delta \rangle} (R) \otimes_A E_A$ with socle generated by $v =  \left[\frac{1}{x_1\cdots x_\delta}\right] \otimes u$.  

Consider now $R^{1/p^e}\otimes_R E_R$, so that $I_e(R)^{1/p^e} = \Ann_{R^{1/p^e}}(1 \otimes v)$. With this observation in place, we claim the following equality of ideals of $R^{1/p^e}$:

\begin{claim}
$I_e(R)^{1/p^e} = \bigl( I_e(A)R + \bigl\langle x_1^{p^e}, \ldots, x_\delta^{p^e} \bigr\rangle \bigr)^{1/p^e}$
\end{claim}
\begin{proof}[Proof of claim]
We may identify $R^{1/p^e} \otimes_R H^\delta_{\langle x_1, \ldots, x_\delta \rangle} (R) = \bigl(H^\delta_{\langle x_1, \ldots, x_\delta \rangle} (R)\bigr)^{1/p^e}$ with $1 \otimes \left[\frac{1}{x_1\cdots x_\delta}\right] \leftrightarrow \left[\frac{1}{x_1^{p^e}\cdots x_\delta^{p^e}}\right]^{1/p^e} $. Using that
\[
\left( R \bigl/ \bigl\langle x_1^{p^e}, \ldots, x_\delta^{p^e} \bigr\rangle \right)^{1/p^e} \xrightarrow{1 \mapsto \left[\frac{1}{x_1^{p^e}\cdots x_\delta^{p^e}} \right]^{1/p^e}} \left( H^\delta_{\langle x_1, \ldots, x_\delta \rangle}(R) \right)^{1/p^e}
\]
is pure as an inclusion of $A^{1/p^e}$-modules, this gives further identifications
\[
\begin{array}{ccc}
R^{1/p^e}\otimes_R E_R & = & \left( H^\delta_{\langle x_1, \ldots, x_\delta \rangle}(R) \right)^{1/p^e} \otimes_{A^{1/p^e}} \left( A^{1/p^e} \otimes_A E_A\right)

\\

1 \otimes v & \leftrightarrow & \left[\frac{1}{x_1^{p^e}\cdots x_\delta^{p^e}} \right]^{1/p^e} \otimes (1 \otimes u)

\\ \mbox{} &&
\\

& \supseteq & \left( R \bigl/ \bigl\langle x_1^{p^e}, \ldots, x_\delta^{p^e} \bigr\rangle \right)^{1/p^e} \otimes_{A^{1/p^e}} \left( A^{1/p^e} \otimes_A E_A\right)

\\
 & \leftrightarrow & 1 \otimes (1 \otimes u).
\end{array}
\]

In particular, $\Ann_{R^{1/p^e}}(1 \otimes v) = \Ann_{R^{1/p^e}} \big(1 \otimes (1 \otimes u)\big)$. It is worth noting that this ideal contains the kernel of $R^{1/p^e} \to  \Bigl( R \bigl/ \bigl\langle x_1^{p^e}, \ldots, x_\delta^{p^e} \bigr\rangle \Bigr)^{1/p^e}$ and its image along this quotient homomorphism is the same as the annihilator of $1 \otimes (1 \otimes u)$ over $\left( R \bigl/ \bigl\langle x_1^{p^e}, \ldots, x_\delta^{p^e} \bigr\rangle \right)^{1/p^e}$.

On the other hand, since $A^{1/p^e} \to  \Bigl( R \bigl/ \bigl\langle x_1^{p^e}, \ldots, x_\delta^{p^e} \bigr\rangle \Bigr)^{1/p^e} $ is flat, it follows that the annihilator of $1 \otimes (1 \otimes u)$ over $\left( R \bigl/ \bigl\langle x_1^{p^e}, \ldots, x_\delta^{p^e} \bigr\rangle \right)^{1/p^e} $ is the extension  of 
\[
\Ann_{A^{1/p^e}}\left(1 \otimes u \in A^{1/p^e}\otimes_A E_A\right)=I_e(A)^{1/p^e}.
\]

Putting everything together proves the claim.
\end{proof}

 
Thus, using the flatness of $A \to R \bigl/ \bigl\langle x_1^{p^e}, \ldots, x_\delta^{p^e} \bigr\rangle $ once again, it follows
\begin{align*}
\ell_R\left(\frac{R}{I_e(R)}\right) &= \ell_R\left(\frac{R}{ I_e(A)R + \bigl\langle x_1^{p^e}, \ldots, x_\delta^{p^e} \bigr\rangle}\right) = \ell_A \left( \frac{A}{I_e(A)}\right) \ell_R \left( \frac{R}{\fram R + \bigl\langle x_1^{p^e}, \ldots, x_\delta^{p^e} \bigr\rangle}\right) \\
&=  p^{e \delta} \ell_A \left( \frac{A}{I_e(A)}\right)\ell_R \left( \frac{R}{\fram R + \langle x_1, \ldots, x_\delta \rangle}\right) = p^{e \delta} \ell_A \left( \frac{A}{I_e(A)}\right).
\end{align*}
Since $\dim R = \dim A + \delta$, the desired equality now follows after dividing by $p^{e \dim R}$ and taking limits.
\end{proof}

We now generalize the above proof to the context of pairs.  We break off the main technical step into a lemma.

\begin{lemma}
\label{lem.IeTransformADtoRD}
Suppose that $(A, \fram) \subseteq (R, \frn)$ is a flat local extension of normal local rings of characteristic $p > 0$ and write $f : \Spec R \to \Spec A$ for the induced map.
For any effective Weil $D$ on $\Spec A$ and $e > 0$, define
\begin{align*}
I_e(A, D) & = \langle a \in A \mid A \to A(D)^{1/p^e} \mbox{ with } 1 \mapsto a^{1/p^e} \mbox{ is not $A$-pure} \rangle \\
I_e(R, f^* D) & = \langle r \in R \mid R \to R(f^* D)^{1/p^e} \mbox{ with } 1 \mapsto r^{1/p^e} \mbox{ is not $R$-pure} \rangle.
\end{align*}

Then if $R/\fram R$ is regular,
\[
\ell_R \left( \frac{R}{I_e(R, f^* D)} \right) =  p^{e(\dim R - \dim A)} \cdot \ell_A\left( \frac{A}{I_e(A, D)} \right).
\]
\end{lemma}

\begin{proof}
If $x_1, \ldots, x_\delta \in R$ give a regular system of parameters of $R / \fram R$,  we have  that $E_R = H^\delta_{\langle x_1, \ldots, x_\delta \rangle} (R) \otimes_A E_A$ with socle generated by $v = \left[\frac{1}{x_1\cdots x_\delta}\right] \otimes u$.  Consider now $R(f^*D)^{1/p^e}\otimes_R E_R$, so that $I_e(R, f^*D)^{1/p^e} = \Ann_{R^{1/p^e}}(1 \otimes v)$.  Using that $R(f^*D) = R\otimes_A A(D)$ and the same identifications made in the proof above, we see that
\[
\begin{array}{ccc}
R(f^*D)^{1/p^e}\otimes_R E_R & = & \left( H^\delta_{\langle x_1, \ldots, x_\delta \rangle}(R) \right)^{1/p^e} \otimes_{A^{1/p^e}} \left( A(D)^{1/p^e} \otimes_A E_A\right)

\\

1 \otimes v & \leftrightarrow & \left[\frac{1}{x_1^{p^e}\cdots x_\delta^{p^e}} \right]^{1/p^e} \otimes (1 \otimes u)

\\ \mbox{} &&
\\

& \supseteq &\Bigl( R \bigl/ \bigl\langle x_1^{p^e}, \ldots, x_\delta^{p^e} \bigr\rangle \Bigr)^{1/p^e}  \otimes_{A^{1/p^e}} \left( A(D)^{1/p^e} \otimes_A E_A\right)

\\
 & \leftrightarrow & 1 \otimes (1 \otimes u).
\end{array}
\]
But since $A^{1/p^e} \to  \Bigl( R \bigl/ \bigl\langle x_1^{p^e}, \ldots, x_\delta^{p^e} \bigr\rangle \Bigr)^{1/p^e} $ is flat, it follows that the annihilator of $1 \otimes (1 \otimes u)$ over $\Bigl( R \bigl/ \bigl\langle x_1^{p^e}, \ldots, x_\delta^{p^e} \bigr\rangle \Bigr)^{1/p^e} $ is the expansion  of $\Ann_{A^{1/p^e}}\left(1 \otimes u \in A(D)^{1/p^e}\otimes_A E_A\right)  = I_e(A,D)^{1/p^e}$.  In other words, we have shown $I_e(R, f^*D)^{1/p^e} = \bigl( I_e(A,D)R + \bigl\langle x_1^{p^e}, \ldots, x_\delta^{p^e} \bigr\rangle \bigr)^{1/p^e}$.  Thus, using the flatness of $A \to  R\bigl/\bigl\langle x_1^{p^e}, \ldots, x_\delta^{p^e} \bigr\rangle$ once again, it follows
\begin{align*}
\ell_R\left(\frac{R}{I_e(R,f^*D)}\right) &= \ell_R\left(\frac{R}{ I_e(A,D)R + \bigl\langle x_1^{p^e}, \ldots, x_\delta^{p^e} \bigr\rangle}\right) \\
&= \ell_A \left( \frac{A}{I_e(A,D)}\right) \ell_R \left( \frac{R}{\fram R + \bigl\langle x_1^{p^e}, \ldots, x_\delta^{p^e} \bigr\rangle}\right) \\
& = p^{e \delta} \ell_A \left( \frac{A}{I_e(A,D)}\right)\ell_R \left( \frac{R}{\fram R + \langle x_1, \ldots, x_\delta \rangle}\right) = p^{e \delta} \ell_A \left( \frac{A}{I_e(A,D)}\right)
\end{align*}
as desired.
\end{proof}

We now can prove the main result of the section.

\begin{theorem}
\label{thm.FsignatureStableFlatMapRegularFiber}
Suppose that $(A, \fram) \subseteq (R, \frn)$ is a flat local extension of normal local rings of characteristic $p > 0$ and write $f : \Spec R \to \Spec A$ the induced map.  Suppose further that $\Delta \geq 0$ is a $\bQ$-divisor on $\Spec A$.  Then if $R/\fram R$ is regular, we have
\[
s(A, \Delta, \fram) = s(R, f^* \Delta, \frn).
\]
\end{theorem}
\begin{proof}
We will first apply \autoref{lem.IeTransformADtoRD} to $D = \lfloor p^e \Delta \rfloor$.  We see that $f^* D = f^* \lfloor p^e \Delta \rfloor \leq \lfloor p^e f^* \Delta \rfloor$.  Hence, recalling that $d = \dim R$ and applying both \autoref{lem:perturb}  and \autoref{lem.IeTransformADtoRD},
\begin{align*}
s(R, f^* \Delta) & =  {\lim_{e \rightarrow \infty} {\frac{1}{p^{de}}  \ell_R\Big(R\big/I_e\big(R, \lceil (p^e-1) f^* \Delta\rceil \big)\Big)}} = {\lim_{e \rightarrow \infty} {\frac{1}{p^{de}}  \ell_R\Big(R\big/I_e\big(R, \lfloor p^e f^* \Delta\rfloor \big)\Big)}} \\
& \leq  {\lim_{e \rightarrow \infty} {\frac{1}{p^{de}} \cdot \ell_R\Big(R\big/I_e\big(R, f^* \lfloor p^e \Delta\rfloor \big)\Big) }} =  { \lim_{e \rightarrow \infty} {\frac{1}{p^{(d - \delta)e}} \ell_A\Big(A\big/I_e\big(A, \lfloor p^e \Delta \rfloor\big)\Big) }}\\
& = s(A, \Delta).
\end{align*}
 On the other hand, if we choose $D = \lceil p^e \Delta \rceil$, then $f^* D = f^* \lceil p^e \Delta \rceil \geq \lceil p^e f^* \Delta \rceil$ and arguing as above gives
$s(R, f^* \Delta) \geq s(A, \Delta)$.
This completes the proof.
\end{proof}

{
We also address Hilbert--Kunz multiplicity under flat extensions (with regular fibers).  There is little work to do here since Kunz proved the result (even before a limit is taken).  
\begin{theorem}\textnormal{(\cite[Proposition 3.9b]{KunzOnNoetherianRingsOfCharP})}
\label{thm.HKRegularFibers}
Let $(A, \fram) \hookrightarrow (R, \fran)$ be a flat local extension of rings of positive characteristic.  Further suppose that $R/\fram R$ is regular.  Then
\[
\eHK(A) =\eHK(R).
\]
\end{theorem}


}

\section{$F$-signature and Hilbert--Kunz multiplicity of general fibers}
\label{sec.FsignatureOfGeneralFibers}


Before proving Bertini-type theorems, we need one more result.  We need to show that if $A \subseteq R$ is a finite type extension of rings such that the perfectified generic fiber has $F$-signature greater than $\lambda$, then so do \emph{most} of the closed fibers.

\begin{setting}
\label{set.A2Setting}
We assume that $A \subseteq R$ is a flat finite type morphism of Noetherian $F$-finite integral domains with fraction fields $K = \Frac(A) \subseteq L = \Frac(R)$.  Suppose further that $A$ is regular and that $A \subseteq R$ has geometrically normal fibers.  Further assume that $\Delta \geq 0$ is a $\bQ$-divisor on $\Spec R$ whose support does not contain any fiber.
\end{setting}
We will not universally assume this setting in this section, but we will always be able to reduce to it.
In order to motivate the main result of this section, we first give an easy proof of a weaker statement. 

\begin{proposition}
\label{prop.A2ForVeryGeneral}
In the setting of \autoref{set.A2Setting}, further suppose that $A$ is finite type over an uncountable algebraically closed field of characteristic $p > 0$.  If
\[
s(R_{K^{\infty}, x}) \geq \lambda
\]
for all $x \in \Spec R_{K^{\infty}}$, then for a very general\footnote{Meaning outside a countable union of proper closed subsets of $\Spec A$} closed point $Q \in \Spec A$ with residue field $k(Q)$,
\[
s(R_{k(Q),x}) \geq \lambda
\]
for all $x \in \Spec R_{k(Q)}$.
\end{proposition}
\begin{proof}
By \cite[Theorem 4.13]{DeStefaniPolstraYaoGlobalizingFinvariants}, for each $e > 0$, and by \autoref{lem.SplitAtPerfectionImpliesSplit} below, we can spread out our splitting and obtain some $a_e, d_e$ and $0 \neq g_e \in A$ so that there is a surjection
\begin{equation}
\label{eq.SurjectivityAfterSpreadingOut}
R^{1/p^{e}}_{A[1/g_e]^{1/p^{e+d_e}}} \to R_{A[1/g_e]^{1/p^{e+d_e}}}^{\oplus a_e}
\end{equation}
and so that
\[\lambda \leq \min_{x \in \Spec R_{K^{\infty}}} \{s(R_{K^{\infty},x})\} = \lim_{e \rightarrow \infty} {\frac{a_e}{p^{e \dim R}}}.\]

Since our $Q$ is very general, $Q \notin V(g_e)$ for any $e$.  Hence we have surjections $A[1/g_e] \to k(Q)$ for all $e$.  We now apply
\[
\blank \otimes_{{A[1/g_e]}^{1/p^{e+d_e}}} k(Q)^{1/p^{e+d_e}}
\]
to \autoref{eq.SurjectivityAfterSpreadingOut} which yields a surjective map
\[
R^{1/p^{e}}_{k(Q)^{1/p^{e+d_e}}} \to R_{k(Q)^{1/p^{e+d_e}}}^{\oplus a_e}.
\]
But $k(Q)$ is perfect and so this can be identified with a surjective map
\[
(R_{k(Q)})^{1/p^{e}} \to R_{k(Q)}^{\oplus a_e}.
\]
The result follows.
%
\end{proof}

\begin{remark}
We do not expect this result to hold for simply general fibers; see \cite{Monskypointquartics1998} for an example where the analagous Hilbert--Kunz statement for general fibers does not hold.
\end{remark}

We now need the following result of P\'erez, the third author, and Yao.

\begin{theorem}
\cite{PerezYaoTucker}
\label{thm.uniformbound}
For every Noetherian ring $A$ of characteristic $p > 0$, and every finitely generated $A$-algebra $R$, and every finitely generated $R$-module $M$, there exists a positive constant $C$ with the following property: for all primes $Q \in \Spec(A)$, all regular $k(Q)$-algebras $\Gamma$, and all $P \in \Spec(R_\Gamma \coloneqq R\otimes_A \Gamma)$, and all $e \geq 1$ , we have that
$\ell_{R_{\Gamma}}\bigl((M_\Gamma)_P / P^{[p^e]}(M_\Gamma)_P \bigr) \leq C p^{e \dim(M_\Gamma)}$ where $M_\Gamma \coloneqq M \otimes_A \Gamma$.
\end{theorem}

The next result is the technical heart of the section.  We state and prove it first in the non-pairs setting and then explain how to generalize it to pairs in a proposition which follows it. 
\begin{notation}
Recall again that if $S$ is a local ring, $a_e(S)$ denotes the free rank of $S^{1/p^e}$ whereas $b_e(S)$ denotes its minimal number of generators. Moreover, if $\Delta$ is a $\bQ$-divisor on $\Spec S$ we shall denote by $a_e^{\Delta}(S)$ the number $a_e^{\lceil (p^e-1) \Delta\rceil}(S)$ defined in \autoref{eqn.Definition:a_e^D}; which is a slight abuse of notation.
\end{notation}

\begin{proposition}
\label{prop.uniformconverginfamily}
Suppose we are in the setting of \autoref{set.A2Setting}. There exists a positive constant $C$ and $0 \neq g \in A$ with the following property: for all $Q \in \Spec\bigl(B \coloneqq A[g^{-1}]\bigr)$, all $d > 0$, all $x \in \Spec \big( R_{k(Q)^{1/p^d}} \big)$, and all $e > 0$,  we have
\[
\left| s\bigl(R_{k(Q)^{1/p^d},x}\bigr) - \frac{a_e\bigl(R_{k(Q)^{1/p^d},x}\bigr)}{\rank_{R_{k(Q)^{1/p^d},x}}\bigl(R_{k(Q)^{1/p^d},x}\bigr)^{1/p^e}} \right| \leq \frac{C}{p^e}.
\]
{
Similarly we have
\[
\left| \eHK\bigl(R_{k(Q)^{1/p^d},x}\bigr) - \frac{b_e\bigl(R_{k(Q)^{1/p^d},x}\bigr)}{\rank_{R_{k(Q)^{1/p^d},x}}\bigl(R_{k(Q)^{1/p^d},x}\bigr)^{1/p^e}} \right| \leq \frac{C}{p^e}.
\]
}
\end{proposition}

\begin{proof}
Let $\delta = \dim R_K = \dim R_{K^{\infty}} $, so that $\rank_{R_{K^{\infty}}} (R_{K^{\infty}})^{1/p^e} = p^{e \delta}$.  Since $R_{A^{1/p^{e+d}}} \to R_{A^{1/p^{e+d}}}^{1/p^e}$ base changes to $R_{K^{\infty}} \to R_{K^{\infty}}^{1/p^e}$ for any $e, d > 0$, we see that
\[
\rank_{R_{A^{1/p^{e+d}}}} R_{A^{1/p^{e+d}}}^{1/p^e} = p^{e \delta}
\]
as well.
Note that $A^{1/p^d} \subseteq R_{A^{1/p^d}}$ is also flat, and for any $Q \in \Spec(A)$ and $x \in \Spec\big(R_{k(Q)^{1/p^d}}\big)$, we have that $\height x - \height Q = \dim R_{k(Q)^{1/p^d},x}$.  Using that $\Frac\big(R_{A^{1/p^{e+d}}}\big) = L_{K^{1/p^{e+d}}}$ as $R_{A^{1/p^{e + d}}}$ is a domain, we also compute
\begin{align} \label{eqn.RefereeSuggestion}
p^{e\delta}  = \rank_{R_{A^{1/p^{e + d}}}}R^{1/p^{e}}_{A^{1/p^{e + d}}} = \Big[L^{1/p^e}_{K^{1/p^{e+d}}}: L_{K^{1/p^{e+d}}}\Big] &= \frac{\Big[L^{1/p^e}_{K^{1/p^{e + d}}}: L_{K^{1/p^{d}}}\Big]}{\Big[L_{K^{1/p^{e+d}}}:L_{K^{1/p^d}}\Big]} \\ \nonumber
& = \frac{\Big[L^{1/p^e}_{K^{1/p^{e + d}}}: L_{K^{1/p^{d}}}\Big]}{\big[K^{1/p^{e}}:K\big]}\\ \nonumber
& = \frac{\big[k(x)^{1/p^e}:k(x)\big] }{\big[k(Q)^{1/p^e}:k(Q)\big]} \cdot p^{e(\height x - \height Q)}\\
& = \frac{\big[k(x)^{1/p^e}:k(x)\big] }{\big[k(Q)^{1/p^e}:k(Q)\big]} \cdot p^{e  \dim R_{k(Q)^{1/p^d},x}} \nonumber
\end{align}

whence
\[
\big[k(Q)^{1/p^e}:k(Q)\big] \cdot p^{e\delta} = \rank_{R_{k(Q)^{1/p^d},x}}\big(R_{k(Q)^{1/p^d},x}\big)^{1/p^e}.
\]

Form exact sequences
\begin{equation}
\label{eq:sescee}
 (R_{A^{1/p}})^{\oplus p^{\delta}} \xrightarrow{\alpha_1} R^{1/p} \to M_1 \to 0
 \end{equation}
 \begin{equation}
 \label{eq:sesmap}
  R^{1/p} \xrightarrow{\alpha_2} (R_{A^{1/p}})^{\oplus p^{\delta}}  \to M_2 \to 0
\end{equation}
of $R_{A^{1/p}}$-modules so that both $M_1, M_2$ are torsion. Take $0 \neq c \in R_{A^{1/p}}$ that kills both; replacing $c$ with $c^p$ if necessary, we may further assume $0 \neq c \in R$. The image of $U = \Spec R[1/c] \subseteq \Spec R$ in $\Spec A$ is open \cite[Tag 01UA]{stacks-project} and contains the image of the generic point. Thus, after inverting an element of $A$, we may assume $c$ does not vanish along any fiber.  In other words, for any $Q \in \Spec A$ and $x \in \Spec R_{k(Q)^{1/p^d}}$, the image of $c$ in $R_{k(Q)}$ is non-zero, and hence also in $R_{k(Q)^{1/p^d},x}$.

Applying $\blank \otimes_{A^{1/p}} k(Q)^{1/p^{d+1}}$  to the sequences above gives that
\[
 \big(R_{k(Q)^{1/p^{d+1}}}\big)^{\oplus p^{\delta}} \to R^{1/p}_{k(Q)^{1/p^{d+1}}} \to M_1 \otimes_{A^{1/p}} k(Q)^{1/p^{d+1}} \to 0
\]
\[
 R^{1/p}_{k(Q)^{1/p^{d+1}}} \to \big(R_{k(Q)^{1/p^{d+1}}}\big)^{\oplus p^{\delta}}  \to M_2 \otimes_{A^{1/p}} k(Q)^{1/p^{d+1}} \to 0
\]
are exact sequences of $R_{k(Q)^{1/p^{d+1}}}$-modules.  But we have that $R_{k(Q)^{1/p^{d+1}}}$ is a free $R_{k(Q)^{1/p^d}}$-module of rank $\big[k(Q)^{1/p}:k(Q)\big]$, so we may view these as sequences
of $R_{k(Q)^{1/p^d}}$-modules and localize at $x \in \Spec R_{k(Q)^{1/p^d}}$ to give the exact sequences of $R_{k(Q)^{1/p^d},x}$-modules
\[
 \big(R_{k(Q)^{1/p^d},x}\big)^{\oplus p^{\delta}[k(Q)^{1/p}:k(Q)]} \xrightarrow{\psi_1} \big(R_{k(Q)^{1/p^d},x}\big)^{1/p} \to \Big(M_1 \otimes_{A^{1/p}} k(Q)^{1/p^{d}}\Big)_x^{\oplus [k(Q)^{1/p}:k(Q)]} \to 0
\]
\[
\big(R_{k(Q)^{1/p^d},x}\big)^{1/p} \xrightarrow{\psi_2} \big(R_{k(Q)^{1/p^d},x}\big)^{\oplus p^{\delta}[k(Q)^{1/p}:k(Q)]}  \to \Big(M_2 \otimes_{A^{1/p}} k(Q)^{1/p^{d}}\Big)_x^{\oplus [k(Q)^{1/p}:k(Q)]} \to 0
\]
so that the summands of the quotients $\big(M_i \otimes_{A^{1/p}} k(Q)^{1/p^{d}}\big)_x$ for $i=1,2$ are killed by the image of $c$ in $R_{k(Q)^{1/p^d},x}$.
If $P$ is the maximal ideal of $R_{k(Q)^{1/p^d},x}$ and $\ell(\blank)$ denotes length over $R_{k(Q)^{1/p^d},x}$,  applying \autoref{thm.uniformbound} (for $A^{1/p} \to R_{A^{1/p}}$ with the $R_{A^{1/p}}$-modules $M_1, M_2$), we have that there is a positive constant $C'$ so that
\begin{align*}
 \ell\left( \frac{\big(M_i \otimes_{A^{1/p}} k(Q)^{1/p^{d}}\big)_x}{P^{[p^e]}\big(M_i \otimes_{A^{1/p}} k(Q)^{1/p^{d}}\big)_x} \right) \leq \frac{C'}{p^e}  p^{e \cdot \dim R_{k(Q)^{1/p^d},x}} = \frac{C'}{p^e} \frac{  [k(x)^{1/p}:k(x)] p^{(e+1) \dim R_{k(Q)^{1/p^d},x}}}{[k(Q)^{1/p}:k(Q)] p^\delta}
\end{align*}
for $i = 1,2$ and all $e > 0$, and where the last equality follows from setting $e=1$ in the computation \autoref{eqn.RefereeSuggestion}.

Let $J_e$ denote {
either the $e$-th Frobenius power $P^{[p^e]}$ of the maximal ideal $P$ used to define the Hilbert--Kunz multiplicity $e_{\mathrm{HK}}(R_{k(Q)^{1/p^d},x})$, or the} $e$-th Frobenius degeneracy ideal $I_e\big(R_{k(Q)^{1/p^d},x}\big)$ used to define the $F$-signature.
Using the well-known properties $(J_e)^{[p]} \subseteq J_{e+1}$ and $\phi\bigl( (J_{e+1})^{1/p}\bigr) \subseteq J_e$
for all $\phi \in \Hom_{R_{k(Q)^{1/p^d},x}}\Big(\big(R_{k(Q)^{1/p^d},x}\big)^{1/p},R_{k(Q)^{1/p^d},x}\Big)$,\footnote{See for instance \cite{PolstraTuckerFSingHKMultCombinedApproach} (and references therein) where these properties are used systematically in the study of Hilbert--Kunz multiplicities and $F$-signatures. } the maps $\psi_1, \psi_2$ induce
\begin{align*}
 &\Big(R_{k(Q)^{1/p^d},x} \big/ J_e\Big)^{\oplus p^{\delta}[k(Q)^{1/p}:k(Q)]} \xrightarrow{\psi_{1,e}} \Big(R_{k(Q)^{1/p^d},x} \big/ J_{e+1}\Big)^{1/p}\\
 &\Big(R_{k(Q)^{1/p^d},x} \big/ J_{e+1}\Big)^{1/p} \xrightarrow{\psi_{2,e}}  \Big(R_{k(Q)^{1/p^d},x} \big/ J_e\Big)^{\oplus p^{\delta}[k(Q)^{1/p}:k(Q)]}
 \end{align*}
with $\coker \psi_{i,e}$ a quotient of $\coker \psi_i$ killed by $P^{[p^e]}$ for $i = 1,2$. Taking lengths and dividing by
\[
\bigl[k(x)^{1/p}:k(x)\bigr] p^{(e+1) \dim R_{k(Q)^{1/p^d},x}}
\]
 gives
  \[
 \left| \frac{\ell \left( \displaystyle \frac{R_{k(Q)^{1/p^d},x}}{J_e} \right)}{p^{e \dim R_{k(Q)^{1/p^d},x}}} -  \frac{\ell \left( \displaystyle \frac{R_{k(Q)^{1/p^d},x}}{J_{e+1}} \right)}{p^{(e+1)\dim R_{k(Q)^{1/p^d},x}}} \right| \leq \frac{C'}{p^{e}} \cdot \frac{1}{[k(Q)^{1/p}:k(Q)] p^{\delta}}
 \]
so that the proposition follows from \cite[Lemma 3.5]{PolstraTuckerFSingHKMultCombinedApproach} with $C = 2C'/[k(Q)^{1/p}:k(Q)]p^\delta$.
 \end{proof}

As mentioned above, we need to generalize the above to the context of pairs.

 \begin{proposition}
\label{prop.uniformconverginfamilydeltas}
Suppose we are in the setting of \autoref{set.A2Setting}. There exists a positive constant $C$ and $0 \neq g \in A$ with the following property: for all $Q \in \Spec(B \coloneqq A[g^{-1}])$, all $d > 0$, all $x \in \Spec \big(R_{k(Q)^{1/p^d}}\big)$, and all $e > 0$,  we have
\[
\left| s\big(R_{k(Q)^{1/p^d},x}, \Delta_{Q,d}\big) - \frac{a_e^{\Delta_{Q,d}}\big(R_{k(Q)^{1/p^d},x}\big)}{\rank_{R_{k(Q)^{1/p^d},x}}\big(R_{k(Q)^{1/p^d},x}\big)^{1/p^e}} \right| \leq \frac{C}{p^e}
\]
where $\Delta_{Q,d} = \Delta|_{\Spec\big(R_{k(Q)^{1/p^d}}\big)}$.
\end{proposition}

\begin{proof}
The desired result follows the argument  in \autoref{prop.uniformconverginfamily}, with modifications we now describe to account for the addition of $\Delta$.  Choose $0 \neq c' \in R$ so that $\Div_R(c') \geq p \Delta$.  After inverting an element of $A$, we may assume $c'$ does not vanish along any fiber and thus $\Div_R(c')|_{\Spec R_{k(Q)}} \geq p \Delta|_{\Spec R_{k(Q)}}$ on fibers as well. In particular, for any $\phi \in \Hom_{R_{k(Q)^{1/p^d},x}}\Big(\big(R_{k(Q)^{1/p^d},x}\big)^{1/p},R_{k(Q)^{1/p^d},x}\Big)$ and $\psi(\blank) = \phi\bigl( (c')^{1/p} \cdot \blank\bigr)$, we have that $\Delta_\psi \geq \Delta_{Q,d}$
and $\Div_R(c')|_{\Spec R_{k(Q)^{1/p^d}}} \geq p \Delta_{Q,d}$ where $\Delta_{Q,d} = \Delta|_{\Spec R_{k(Q)^{1/p^d}}}$.

Replace $\alpha_1, \alpha_2$ in the right exact sequences \ref{eq:sescee} and \ref{eq:sesmap} with their premultiples
\[
(R_{A^{1/p}})^{\oplus p^\delta} \xrightarrow{\cdot c'} (R_{A^{1/p}})^{\oplus p^\delta} \xrightarrow{\alpha_1} R^{1/p}
\]
\[
R^{1/p} \xrightarrow{\cdot (c')^{1/p}} R^{1/p} \xrightarrow{\alpha_2}  (R_{A^{1/p}})^{\oplus p^\delta},
\]
respectively.  In \cite[proof of Theorem 4.12]{PolstraTuckerFSingHKMultCombinedApproach}, the properties
\[
\begin{array}{c} \displaystyle
c'\Big(I_e\big(R_{k(Q)^{1/p^d},x}, \lceil (p^e-1)\Delta_{Q,d}\rceil \big)^{[p]}\Big) \subseteq I_{e+1}\big(R_{k(Q)^{1/p^d},x}, \lceil (p^{e+1}-1)\Delta_{Q,d}\rceil \big) \\
\displaystyle \phi\Big( \big(c'I_{e+1}\big(R_{k(Q)^{1/p^d},x},\lceil (p^{e+1}-1)\Delta_{Q,d}\rceil\big)\big)^{1/p}\Big) \subseteq I_e\big(R_{k(Q)^{1/p^d},x},\lceil (p^e-1)\Delta_{Q,d}\rceil\big)
\end{array}
\]
are shown to hold.  The proof of  \autoref{prop.uniformconverginfamily} can now be traced through without further modification.
The corresponding maps $\psi_i$ satisfy the analogs of the above properties with respect to the ideals $I_e\big(R_{k(Q)^{1/p^d},x}, \lceil \big(p^e-1)\Delta_{Q,d}\rceil\big)$ and pass to maps $\psi_{i,e}$ on the quotients, with $\coker \psi_{i,e}$ a quotient of $\coker \psi_i$ killed by $P^{[p^e]}$ for $i = 1,2$.  In particular, the constant $C'$ derived in the proof of \autoref{prop.uniformconverginfamily} from \autoref{thm.uniformbound} once more gives
\begin{align*}
 &\left| \frac{\ell \left( \displaystyle \frac{R_{k(Q)^{1/p^d},x}}{I_e\big(R_{k(Q)^{1/p^d},x}, \lceil (p^e-1)\Delta_{Q,d}\rceil\big)} \right)}{p^{e \cdot\dim R_{k(Q)^{1/p^d},x}}} -  \frac{\ell \left( \displaystyle \frac{R_{k(Q)^{1/p^d},x}}{I_{e+1}\big(R_{k(Q)^{1/p^d},x}, \lceil (p^{e+1}-1)\Delta_{Q,d}\rceil \big)} \right)}{p^{(e+1)\dim R_{k(Q)^{1/p^d},x}}} \right|\\
 \leq &{}\frac{C'}{p^{e}} \cdot \frac{1}{[k(Q)^{1/p}:k(Q)] p^{\delta}}
 \end{align*}
so that once more the proposition follows from \cite[Lemma 3.5]{PolstraTuckerFSingHKMultCombinedApproach} with $C = 2C'/[k(Q)^{1/p}:k(Q)]p^\delta$.
\end{proof}


\begin{lemma}
\label{lem.SplitAtPerfectionImpliesSplit}
In the setting of \autoref{set.A2Setting}, suppose that there is a surjective $R_{K^{\infty}}$-linear map
\begin{equation}
\label{eq.BaseChangeSurjectionImplies}
(R_{K^{\infty}})^{1/p^e} \to R_{K^{\infty}}^{\oplus a_e}
\end{equation}
for some $a_e > 0$.
Then for some $d_e > 0$ and $0 \neq g \in A$, setting $B = A[1/g]$, there is a surjective $R_{B^{1/p^{e+d_e}}}$-linear map
\[
R^{1/p^e} \otimes_B B^{1/p^{e+d_e}} = R^{1/p^{e}}_{B^{1/p^{e+d_e}}} \to R_{B^{1/p^{e+d_e}}}^{\oplus a_e}
\]
which tensors with $\otimes_{B^{1/p^{e+d_e}}} K^{\infty}$ to recover \autoref{eq.BaseChangeSurjectionImplies}.

Furthermore, suppose there exists a Weil divisor $\Delta$ on $\Spec R$ (still in \autoref{set.A2Setting}) such that each component projection $\rho : R_{K^{\infty}}^{1/p^e} \to R_{K^{\infty}}$ corresponds to a $\bQ$-divisor $\Delta_{\rho} \geq \xi^* \Delta$ (where $\xi : \Spec R_{K^{\infty}} \to \Spec R$ is the canonical map).  In this case we can choose our $g$ such that the map
\[
R^{1/p^{e}}_{B^{1/p^{e+d_e}}} \to R_{B^{1/p^{e+d_e}}}^{\oplus a_e}
\]
also has the property that each component projection $\gamma : R^{1/p^{e}}_{B^{1/p^{e+d_e}}} \to R_{B^{1/p^{e+d_e}}}$ corresponds to a $\bQ$-divisor $\Delta_{\gamma}$ on $\Spec R_{B^{1/p^{d_e}}}$ such that $\Delta_{\gamma} \geq \eta^* \Delta$ (where $\eta : \Spec R_{B^{1/p^{d_e}}} \to \Spec R$ is the canonical map).
\end{lemma}
\begin{proof}
First notice since we are planning to invert an element of $A$, we may assume that $\omega_A \cong A$.  Furthermore, any future $B$ satisfies the same property.  Note also that $R_{K^{\infty}}$ is a normal domain by \autoref{lem.BaseChangeIsNormalInSetting}.
We have $(R_{K^{\infty}})^{1/p^e} \cong R^{1/p^e} \otimes_{K^{1/p^e}} K^{\infty}$ and so we can view our initial map as an $R_{K^{\infty}}$-linear map, and in particular a $K^{\infty}$-linear map
\[
\phi : (R^{1/p^e})_{K^{\infty}} \to R_{K^{\infty}}^{\oplus a_e}.
\]
In other words, we are simply identifying relative and absolute Frobenius over a perfect field.  Fix $x_1, \dots, x_t$ a generating set for $R^{1/p^e}$ over $R_{A^{1/p^e}}$.  By base change, the images of those elements are also a generating set for $R^{1/p^e}_{K^{\infty}}$ over $R_{K^{\infty}}$ or for any intermediate base change.   
We may assume that all of the $\phi(x_i)$ land inside $R_{K^{1/p^{e+d_e}}}^{\oplus a_e} \hookrightarrow R_{K^{\infty}}^{\oplus a_e}$ for some $d_e > 0$.  Note the $\phi(x_i)$ generate $\phi\left(R^{1/p^e}_{K^{1/p^{e+d_e}}}\right)$ as a $R_{K^{1/p^{e+d_e}}}$-module.
This implies that
\[
\phi\left(R^{1/p^e}_{K^{1/p^{e+d_e}}}\right) \subseteq R_{K^{1/p^{e+d_e}}}^{\oplus a_e}
\]
and hence we have a map (which we also call $\phi$)
\[
\phi : R^{1/p^e}_{K^{1/p^{e+d_e}}} \to R_{K^{1/p^{e+d_e}}}^{\oplus a_e}.
\]
Since this map becomes surjective after the faithfully flat base change to $K^{\infty}$, it is surjective.

By the same argument as above, we may find a denominator $g'$ so that
\[
\phi\left(R^{1/p^e}_{A^{1/p^{e+d_e}}[1/g']}\right) \subseteq R_{A^{1/p^{e+d_e}}[1/g']}^{\oplus a_e}
\]
which produces a map
\[
\phi : R^{1/p^e}_{A^{1/p^{e+d_e}}[1/g']} \to R_{A^{1/p^{e+d_e}}[1/g']}^{\oplus a_e}.
\]
We do not know that this map is surjective but the cokernel is zero if we tensor with $R_{K^{1/p^{e+d_e}}}$.
Inverting another element $g''$, setting $g = g' g''$ and $B = A[g^{-1}]$, we may assume that
\[
\phi : R^{1/p^e}_{B^{1/p^{e+d_e}}} \to R_{B^{1/p^{e+d_e}}}^{\oplus a_e}
\]
is surjective as desired.

Now we move on to the statement involving $\Delta$.  We begin in exactly the same way and produce a surjective map
\[
\phi : R^{1/p^e}_{B^{1/p^{e+d_e}}} \to R_{B^{1/p^{e+d_e}}}^{\oplus a_e}
\]
for some $d_e > 0$ where $B=A[1/g]$.  We need to show that the component projection maps $\gamma$ coming from $\phi$ produce divisors $\Delta_{\gamma}$ on $\Spec R_{B^{1/p^{d_e}}}$ via \autoref{subsec.DivisorsAndFamilies} such that $\Delta_{\gamma} \geq \eta^* \Delta$ where $\eta : \Spec R_{B^{1/p^{d_e}}} \to \Spec R$ is the canonical map.  Consider the following diagram where all of these maps are labeled.
\[
\xymatrix{
\Spec R_{K^{\infty}} \ar@/_2pc/[rr]_-{\xi} \ar[r]^-{\zeta} & \Spec R_{B^{1/p^{d_e}}} \ar[r]^-{\eta} & \Spec R
}
\]
We also know that $\zeta^* \Delta_{\gamma} = \Delta_{\rho}$ by \autoref{lem.RelativeDivisorsAndBaseChange} since $\gamma$ base changes to a projection $\rho$.  Since $\Delta_{\rho} \geq \xi^* \Delta = \zeta^* \eta^* \Delta$, we see that $\zeta^* \Delta_{\gamma} \geq \zeta^* \eta^* \Delta$.  Since $\Delta$ has no vertical components neither does $\eta^* \Delta$.  Therefore because $\Delta_{\gamma} \geq 0$, we conclude that $\Delta_{\gamma} \geq \eta^* \Delta$ as desired.
\end{proof}

{ 
\begin{lemma}
\label{lem.HKAtPerfectionImplies}
In the setting of \autoref{set.A2Setting}, suppose that there is a surjective $R_{K^{\infty}}$-linear map
\begin{equation}
\label{eq.BaseChangeSurjectionImpliesAgain}
  R_{K^{\infty}}^{\oplus b_e} \to (R_{K^{\infty}})^{1/p^e}
\end{equation}
for some $b_e > 0$.
Then for some $d_e > 0$ and $0 \neq g \in A$, setting $B = A[1/g]$, there is a surjective $R_{B^{1/p^{e+d_e}}}$-linear map
\[
R_{B^{1/p^{e+d_e}}}^{\oplus b_e} \to
R^{1/p^e} \otimes_B B^{1/p^{e+d_e}} = R^{1/p^{e}}_{B^{1/p^{e+d_e}}}
\]
which tensors with $\otimes_{B^{1/p^{e+d_e}}} K^{\infty}$ to recover \autoref{eq.BaseChangeSurjectionImpliesAgain}.
\end{lemma}
\begin{proof}
The proof strategy is the same as before in \autoref{lem.SplitAtPerfectionImpliesSplit}.
Note that $R_{K^{\infty}}$ is a normal domain by \autoref{lem.BaseChangeIsNormalInSetting}.
We have $(R_{K^{\infty}})^{1/p^e} \cong R^{1/p^e} \otimes_{K^{1/p^e}} K^{\infty}$ and so we can view our initial map as an $R_{K^{\infty}}$-linear map, and in particular a $K^{\infty}$-linear map
\[
\psi : R_{K^{\infty}}^{\oplus b_e}  \to (R^{1/p^e})_{K^{\infty}}.
\]
In other words, we are simply identifying relative and absolute Frobenius over a perfect field.

The images of the standard basis $e_i \in R_{K^{\infty}}^{\oplus b_e}$ form a generating set for $(R^{1/p^e})_{K^{\infty}}$ by hypothesis.  We may assume that $\psi(e_i) \in (R^{1/p^e})_{K^{1/p^{e+d_e}}}$ for some $d_e > 0$.  Now, the $\psi(e_i)$ generate $\psi( R_{K^{1/p^{e+d_e}}}^{\oplus b_e})$ as a $R_{K^{1/p^{e+d_e}}}$-module and so we have a map which we also call $\psi$
\[
\psi : R_{K^{1/p^{e+d_e}}}^{\oplus b_e}  \to (R^{1/p^e})_{K^{1/p^{e+d_e}}}.
\]
Since the faithfully flat base change of this map with $K^{\infty}$ is the other map called $\psi$, this $\psi$ is also surjective. Likewise, we also can find a denominator $g'$ and so induce a map
\[
\psi : R_{A^{1/p^{e+d_e}}[1/g']}^{\oplus b_e}  \to (R^{1/p^e})_{A^{1/p^{e+d_e}}[1/g']}.
\]
Inverting yet another element if necessary, let us assume that this map is also surjective as desired.
\end{proof}
}




\begin{theorem}
\label{thm.GeneralFiberImpliesMostSpecialFibers}
In the setting of \autoref{set.A2Setting}, further suppose that $A$ is finite type over a perfect field of characteristic $p > 0$.
If
\[
s(R_{K^{\infty}}, \Delta_{K^{\infty}}) > \lambda
\]
then there exists an open dense $U \subseteq \Spec A$ such that for any closed point $Q \in U$,
\[
s(R_{k(Q)}, \Delta_{k(Q)}) > \lambda.
\]
\end{theorem}
\begin{proof}
Inverting an element of $A$ if necessary, we may choose a positive constant $C$ as in \autoref{prop.uniformconverginfamilydeltas}.
By \cite{DeStefaniPolstraYaoGlobalizingFinvariants}, fix $0 < \epsilon \ll 1$ such that $s(R_{K^{\infty}, x}, \Delta_{K^{\infty}}) > \lambda + 2\epsilon$ for all $x \in \Spec R_{K^\infty}$. Pick $e \gg 0$ so that $C/p^e < \epsilon$, so that we have
\[
a_e^{\Delta_{K^{\infty}}}\big(R_{K^\infty,x}\big) \Big/ \rank_{R_{K^\infty,x}}\big(R^{1/p^e}_{K^\infty,x}\big) > \lambda + \epsilon.
\]
By \cite[Theorem 4.22]{DeStefaniPolstraYaoGlobalizingFinvariants} and by \autoref{lem.SplitAtPerfectionImpliesSplit}, after inverting an element of $A$ we may assume there is a $d \geq 0$ and a surjective $R_{A^{1/p^{e+d}}}$-linear map
\[
R_{A^{1/p^{e+d}}}^{1/p^e} \to R_{A^{1/p^{e+d}}}^{\oplus a_e} \text{, where $a_e \coloneqq a_e^{\Delta_{K_{\infty}}}(R_{K^\infty})$}
\]
satisfying the divisorial condition on projections from \autoref{lem.SplitAtPerfectionImpliesSplit}.
Applying $\blank \otimes_{A^{1/p^{e+d}}} k(Q)^{1/p^{e+d}}$ for maximal $Q \in \Spec(A)$ gives a surjection
\[
R_{k(Q)^{1/p^{e+d}}}^{1/p^e} \to R_{k(Q)^{1/p^{e+d}}}^{\oplus a_e}.
\]
of $R_{k(Q)^{1/p^{e+d}}}$-modules where still $a_e = a_e^{\Delta_{K_{\infty}}}(R_{K^\infty})$.  Note the projections corresponding to this map also have the property that their corresponding divisors are $\geq \Delta_Q \coloneqq \Delta|_{R_{k(Q)^{1/p^d}}}$ by \autoref{lem.RelativeDivisorsAndBaseChange}.

Since $A$ is finite type over a perfect field and $Q$ is maximal, $k(Q)$ is also perfect and so $k(Q)^{1/p^{e+d}} = k(Q)^{1/p^e} = k(Q)$.  It also follows that
\[
\rank_{R_{K^{\infty}}}(R_{K^{\infty}}^{1/p^e}) = \rank_{R_{k(Q)}}(R_{k(Q)}^{1/p^e})
\]
since $A \subseteq R$ is flat and of finite type and $A$ is $F$-finite.

Therefore we have a surjection
\[
(R_{{k(Q)},x})^{1/p^e} \to R_{k(Q),x}^{\oplus a_e}
\]
showing that 
\[
\frac{a_e^{\Delta_{Q}}(R_{k(Q),x})}{\rank_{R_{k(Q),x}}(R_{k(Q),x})^{1/p^e}} > \lambda + \epsilon.
\]
Thus, it follows once again from \autoref{prop.uniformconverginfamilydeltas} that
\[
s(R_{k(Q),x}, \Delta_Q) > \lambda
\]
for all $x \in \Spec R_{k(Q)}$ as desired.
%
\end{proof}

{
\begin{theorem}
\label{thm.GeneralFiberImpliesMostSpecialFibersHK}
In the setting of \autoref{set.A2Setting}, further suppose that $A$ is finite type over a perfect field of characteristic $p > 0$.
If
\[
\eHK(R_{K^{\infty}}) < \lambda
\]
then there exists an open dense $U \subseteq \Spec A$ such that for any closed point $Q \in U$,
\[
\eHK(R_{k(Q)}) < \lambda.
\]
\end{theorem}
\begin{proof}
Inverting an element of $A$ if necessary, we may choose a positive constant $C$ as in \autoref{prop.uniformconverginfamily}.
By \cite{DeStefaniPolstraYaoGlobalizingFinvariants}, fix $0 < \epsilon \ll 1$ such that $\eHK(R_{K^{\infty}, x}) < \lambda + 2\epsilon$ for all $x \in \Spec R_{K^\infty}$. Pick $e \gg 0$ so that $C/p^e < \epsilon$, so that we have
\[
b_e\big(R_{K^\infty,x}\big) \Big/ \rank_{R_{K^\infty,x}}\big(R^{1/p^e}_{K^\infty,x}\big) < \lambda + \epsilon.
\]
By \autoref{lem.HKAtPerfectionImplies}, after inverting an element of $A$ we may assume there is a $d \geq 0$ and a surjective $R_{A^{1/p^{e+d}}}$-linear map
\[
R_{A^{1/p^{e+d}}}^{\oplus b_e}  \to R_{A^{1/p^{e+d}}}^{1/p^e} \text{, where $b_e \coloneqq b_e(R_{K^\infty})$}.
\]
Applying $\blank \otimes_{A^{1/p^{e+d}}} k(Q)^{1/p^{e+d}}$ for maximal $Q \in \Spec(A)$ gives a surjection
\[
R_{k(Q)^{1/p^{e+d}}}^{\oplus b_e} \to R_{k(Q)^{1/p^{e+d}}}^{1/p^e}
\]
of $R_{k(Q)^{1/p^{e+d}}}$-modules. 

Since $A$ is of finite type over a perfect field and $Q$ is maximal, $k(Q)$ is also perfect and so $k(Q)^{1/p^{e+d}} = k(Q)^{1/p^e} = k(Q)$.  It also follows that
\[
\rank_{R_{K^{\infty}}}\Big(R_{K^{\infty}}^{1/p^e}\Big) = \rank_{R_{k(Q)}}\Big(R_{k(Q)}^{1/p^e}\Big)
\]
since $A \subseteq R$ is flat and of finite type and $A$ is $F$-finite. Therefore we have a surjection
\[
 R_{k(Q),x}^{\oplus b_e} \to (R_{{k(Q)},x})^{1/p^e}
\]
showing that 
\[
b_e(R_{k(Q),x}) \Big/ \rank_{R_{k(Q),x}}\big(R_{k(Q),x}^{1/p^e} \big) < \lambda + \epsilon.
\]
Thus, it follows once again from \autoref{prop.uniformconverginfamily} that
\[
\eHK(R_{k(Q),x}) < \lambda
\]
for all $x \in \Spec R_{k(Q)}$ as desired.
\end{proof}
}

\section{Bertini theorems for $F$-signature and Hilbert--Kunz multiplicity}

In this section, we conclude by proving our Bertini theorems for $F$-signature.  We first recall the main result of \cite{CuminoGrecoManaresiAxiomatic} and the very slight generalization to the context of pairs of \cite{SchwedeZhangBertiniTheoremsForFSings}.

Suppose $\sP$ is a local property for locally Noetherian schemes (respectively pairs $(X, \Delta \geq 0)$).
\begin{itemize}
\item [(A1)]\label{prop.A1}  Whenever $\phi : Y \to Z$ is a flat morphism with regular fibers and $Z$ (resp. $(Z, \Delta)$) is $\sP$, then $Y$ (resp. $(Y, \phi^* \Delta)$) is $\sP$ too.
\item [(A2)]\label{prop.A2}  Let $\phi : Y \to S$ be a morphism of finite type where $Y$ is excellent and $S$ is integral with generic point $\eta$.  If $Y_{\eta}$ (resp. $(Y_{\eta}, \Delta|_{Y_{\eta}}$) is geometrically $\sP$, then there exists an open neighborhood $U$ of $\eta$ in $S$ such that $Y_{s}$ (resp. $(Y_s, \Delta|_{Y_s})$) is geometrically $\sP$ for each $s \in U$.
\item [(A3)]  $\sP$ is open on schemes $X$ (resp. pairs $(X, \Delta)$) of finite type over a field.
\label{prop.A3}
\end{itemize}

\begin{theorem}\textnormal{\cite[Theorem 1]{CuminoGrecoManaresiAxiomatic}}
\label{thm.CuminoGrecoManaresi}
Let $X$ be a scheme of finite type over an algebraically closed field $k$, let $\phi : X \to \bP^n_k$ be a morphism with separable generated residue field extensions.  Suppose $X$ (resp. $(X, \Delta)$) has a property $\sP$ satisfying conditions (A1) and (A2).  Then there exists a nonempty open subscheme $U$ of $(\bP_k^n)^*$ such that $\phi^{-1}(H)$ has property $\sP$ for each hyperplane $H \in U$.
\end{theorem}

\begin{remark}
\label{rem.A2forClosedPointsAndVeryGeneralEnough}
In the proof of \autoref{thm.CuminoGrecoManaresi}, when using (A2), $S$ is (an open subset) of $(\bP_k^n)^*$ and $\phi^{-1}(s) = Y_s$ are fibers that are exactly equal to the hyperplane sections.  In particular, one may additionally assume that $S$ is of finite type over an algebraically closed field and we only need to verify (A2) for the closed fibers.

Suppose that $k = \overline{k}$ is uncountable and consider the following weakening of (A2):
\begin{itemize}
\item[(B2)]  Let $\phi : Y \to S$ be a morphism of finite type where $S$ is integral of finite type over $k$, with generic point $\eta$.  If $Y_{\eta}$ (resp. $(Y_{\eta}, \Delta|_{Y_{\eta}}$) is geometrically $\sP$, then for a very general closed point $s \in S$ we have that $Y_{s}$ (resp. $(Y_s, \Delta|_{Y_s})$) is geometrically $\sP$ for each $s \in U$.
    \label{prop.B2}
\end{itemize}
If (A1) and (B2) hold for $\sP$, then it immediately follows that the weakening of \autoref{thm.CuminoGrecoManaresi} holds for \emph{very general} hyperplane sections.
\end{remark}

\begin{corollary}\textnormal{\cite[Corollary 2]{CuminoGrecoManaresiAxiomatic}}
\label{cor.OpenSubsetRestriction}
Let $k = \overline{k}$, $V \subseteq \bP^n_k$ be a closed subscheme (resp. and let $\Delta$ be a $\bQ$-divisor on $V$) and let $\sP$ be a local property satisfying (A1).
\begin{enumerate}
\item If $\sP$ satisfies (A2), and $V$ (resp. $(V, \Delta)$) is $\sP$, then the general hyperplane section of $V$ (resp. $(V, \Delta)$) satisfies $\sP$.
\item If $k$ is uncountable, $\sP$ satisfies (B2), and $V$ (resp. $(V, \Delta)$) is $\sP$, then the very general hyperplane section of $V$ (resp. $(V, \Delta)$) satisfies $\sP$.
\item Suppose $\sP$ satisfies (A1), (A2) and (A3), and set $\sP(V)$ to be the $\sP$ locus of $V$ then
\[
\sP(V \cap H) \supseteq \sP(V) \cap H
\]
for a general hyperplane $H$.
\end{enumerate}
\end{corollary}

Combining this machinery with our work of the previous sections, we immediately obtain the main results of the paper.  We first state the result for $F$-signature.

\begin{theorem}
\label{thm.Main}
Suppose that $\psi : X \to \bP^n_k$ is a map of varieties over $k = \overline{k}$ with separably generated residue field extensions (for example, a closed embedding) and that $\Delta \geq 0$ is a $\bQ$-divisor on $X$.  Suppose that $\lambda \geq 0$.
\begin{enumerate}
\item \label{thm.Main.1}  Suppose $s(\O_{X,x}, \Delta_x) > \lambda$ for all closed points $x \in X$. Choose a general hyperplane $H \subseteq \bP^n_k$, and set $Y = \psi^{-1}(H)$.  Then
    \[
        s(\O_{Y,y}, \Delta_y|_Y) > \lambda
    \]
    for all closed points $y \in Y$.
\item \label{thm.Main.2} Suppose $\psi : X \subseteq \bP^n_k$ is a closed embedding.  Let $U_X \subseteq X$ be the subset of points $x \in X$ such that $s(\O_{X,x}, \Delta_x) > \lambda$.  For any hyperplane $H \subseteq \bP^n_k$ let $U_{H \cap X}$ denote the set of points $x \in X \cap H$ such that $s(\O_{H,x}, \Delta_x|_H) > \lambda$.  Then for $H$ a general hyperplane in $\bP^n_k$
    \[
        U_{H \cap X} \supseteq U_X \cap H.
    \]
\item \label{thm.Main.3} Suppose additionally that $k$ is uncountable, and that $s(\O_{X,x}, \Delta_x) \geq \lambda$ for all closed points $x \in X$.  Choose a very general hyperplane $H \subseteq \bP^n_k$, and set $Y = \psi^{-1}(H)$.  Then
    \[
        s(\O_{Y,y}, \Delta_y|_Y) \geq \lambda
    \]
    for all closed points $y \in Y$.
\end{enumerate}
\end{theorem}
\begin{proof}
For part \autoref{thm.Main.1}, we consider the condition $\sP$ that $s(\O_{X,x}, \Delta) > \lambda$.  We apply \autoref{thm.CuminoGrecoManaresi} using the fact that properties (A1) and (A2) are satisfied by \autoref{thm.FsignatureStableFlatMapRegularFiber} and \autoref{thm.GeneralFiberImpliesMostSpecialFibers} respectively (see also \autoref{rem.A2forClosedPointsAndVeryGeneralEnough}).

For part \autoref{thm.Main.2}, we simply use \autoref{cor.OpenSubsetRestriction} and use the fact that $s(\O_{X,x}, \Delta) > \lambda$ is an open condition by semicontinuity \cite{PolstraFsigSemiCont,PolstraTuckerFSingHKMultCombinedApproach} so that (A3) is satisfied.

Part \autoref{thm.Main.3} either follows from \autoref{thm.Main.1} by considering a sequence of $\lambda_i = \lambda - 1/i$ or alternately can be directly proven via \autoref{rem.A2forClosedPointsAndVeryGeneralEnough} by replacing property (A2) with (B2), which was verified in \autoref{prop.A2ForVeryGeneral}.
\end{proof}

{
\begin{theorem}
\label{thm.MainHK}
Suppose that $\psi : X \to \bP^n_k$ is a map of normal varieties over $k = \overline{k}$ with separably generated residue field extensions (for example, a closed embedding).  Suppose that $\lambda \geq 1$.
\begin{enumerate}
\item \label{thm.MainHK.1}  Suppose $\eHK(\O_{X,x}) < \lambda$ for all closed points $x \in X$. Choose a general hyperplane $H \subseteq \bP^n_k$, and set $Y = \psi^{-1}(H)$.  Then
    \[
        \eHK(\O_{Y,y}) < \lambda
    \]
    for all closed points $y \in Y$.
\item \label{thm.MainHK.2} Suppose $\psi : X \subseteq \bP^n_k$ is a closed embedding.  Let $U_X \subseteq X$ be the subset of points $x \in X$ such that $\eHK(\O_{X,x}) < \lambda$.  For any hyperplane $H \subseteq \bP^n_k$ let $U_{H \cap X}$ denote the set of points $x \in X \cap H$ such that $\eHK(\O_{H,x}) < \lambda$.  Then for $H$ a general hyperplane in $\bP^n_k$
    \[
        U_{H \cap X} \supseteq U_X \cap H.
    \]
\item \label{thm.MainHK.3} Suppose additionally that $k$ is uncountable, and that $\eHK(\O_{X,x}) \leq \lambda$ for all closed points $x \in X$.  Choose a very general hyperplane $H \subseteq \bP^n_k$, and set $Y = \psi^{-1}(H)$.  Then
    \[
        \eHK(\O_{Y,y}) \leq \lambda
    \]
    for all closed points $y \in Y$.
\end{enumerate}
\end{theorem}
\begin{proof}
For part \autoref{thm.MainHK.1}, we consider the condition $\sP$ that $\eHK(\O_{X,x}) < \lambda$.  We apply \autoref{thm.CuminoGrecoManaresi} using the fact that properties (A1) and (A2) are satisfied by \autoref{thm.HKRegularFibers} and \autoref{thm.GeneralFiberImpliesMostSpecialFibersHK} respectively.

For part \autoref{thm.Main.2}, we simply use \autoref{cor.OpenSubsetRestriction} and use the fact that $\eHK(\O_{X,x}) < \lambda$ is an open condition by semicontinuity \cite{SmirnovSemiContHK,PolstraTuckerFSingHKMultCombinedApproach} so that (A3) is satisfied.

Part \autoref{thm.Main.3} follows from \autoref{thm.Main.1} by considering a sequence of $\lambda_i = \lambda - 1/i$.
\end{proof}
}

\begin{remark}
Recently, building on the techniques used in this paper, Datta and Simpson \cite[Theorem 4.1]{datta2020hilbertkunz} have shown that the normality hypothesis on $X$ in \autoref{thm.MainHK} can be weakened.
\end{remark}

\begin{remark}
We expect that the very general hypothesis in (c) above cannot be removed.  Indeed, see \cite{Monskypointquartics1998}.
\end{remark}

\bibliographystyle{skalpha}
\bibliography{MainBib}
\end{document}